\documentclass[DIV=16]{scrartcl}

\usepackage{amsmath,amssymb}
\usepackage{hyperref}
\hypersetup{colorlinks}
\usepackage[utf8]{inputenc}
\usepackage{caption}
\usepackage[thmmarks]{ntheorem}
\usepackage[numbers]{natbib}
\usepackage{stmaryrd}

\usepackage{tikz,pgf}
\usetikzlibrary{angles,quotes,calc}

\newcommand{\cc}{{\ensuremath{\mathbb C}}}
\newcommand{\nn}{{\ensuremath{\mathbb N}}}
\newcommand{\rr}{{\ensuremath{\mathbb R}}}
\newcommand{\zz}{{\ensuremath{\mathbb Z}}}
\newcommand{\qq}{{\ensuremath{\mathbb Q}}}
\DeclareMathOperator{\RE}{Re}

\theoremstyle{plain}

\newtheorem{theorem}{Theorem}
\newtheorem{corollary}[theorem]{Corollary}
\newtheorem{lemma}[theorem]{Lemma}
\newtheorem{definition}[theorem]{Definition}
\newtheorem{proposition}[theorem]{Proposition}

\theorembodyfont{\upshape}
\newtheorem{remark}[theorem]{Remark}
\newtheorem{example}[theorem]{Example}

\theoremstyle{nonumberplain}
\theoremsymbol{\ensuremath{\blacksquare}}
\newtheorem{proof}{Proof.\quad}

\begin{document}
 \title{When is an origami set a ring?}
 \author{Florian Möller\thanks{fmoeller@mathematik.uni-wuerzburg.de}}
 \date{\today}
 \maketitle
 
\begin{abstract}
 \noindent Starting with a flat sheet of paper, points can be constructed as the intersection of two folds. The set of constructible points clearly depends on which folds are admissible. In this paper, we study the situation where a fold is admissible if its slope is admissible and it contains an already constructed or a generator point. We give an explicit characterization of the set of constructible points. We also state several criteria for this set to be a ring. This answers questions originally raised by Erik Demaine and discussed by Butler~et~al.~\cite{but13} and Buhler~et~al.~\cite{buh12}.
 
 \minisec{Acknowledgement}
 
 We are grateful to Dmitri Nedrenco for pointing out this problem to us.
\end{abstract}

\section{Introduction}\label{sec:intro}

Identifying a flat sheet of paper with the Euclidean plane and, subsequently, the Euclidean plane with the field~$\cc$ of complex numbers it is possible to describe geometric constructions with algebraic means. This method of algebraization of geometric problems has proved very successful: For instance, in the language of field theory it is easy to precisely describe the points constructible by compass and straightedge. This readily shows that some of the classical compass-and-straightedge construction problems are unsolvable or that the set of compass-and-straightedge constructible points is a subfield of~$\cc$.

In this paper we apply the idea of algebraization to a type of construction motivated by origami related questions. In~\cite{but13} Butler et al.\@ ask which points in the plane can be constructed using origami techniques when there is the following limitation on the folds: Starting from the generator points $0,1\in\cc$ only folds through already existing points with prescribed slopes are allowed. We call the set of points obtained in this way an \emph{origami set}.

\paragraph{Origami sets and origami rings}

\noindent We use the following mathematical concepts to model the folding process:

A fold is a straight line, its slope the angle enclosed with the real axis. Obviously, it suffices to only consider angles~$\alpha$ with $0\leq\alpha<\pi$.

Fix a subset~$U\subseteq[0,\pi[$. We interpret $U$ as the set of prescribed slopes of lines. Throughout this paper we assume that $0\in U$ and that $U$ contains at least three elements.

The set of generator points is given by $M_0:=\{0,1\}\subseteq\cc$. We define sets $M_k$ recursively: If~$M_{k-1}$ is already known for some $k\in\nn$, then~$M_k$ denotes the set of all intersection points of lines through elements of $M_{k-1}$ with prescribed slopes, i.\,e.
\[M_k:=\bigcup_{\substack{z,z'\in M_{k-1}\\\alpha,\alpha'\in U\text{ with } \alpha\ne\alpha'}} \bigl(z+\rr\cdot\exp(i\alpha)\bigr)\cap\bigl(z'+\rr\cdot\exp(i\alpha')\bigr).\]
Choosing $z=z'$ in the above equation yields $M_{k-1}\subseteq M_k$. The union
\[M(U):=\bigcup_{k=0}^\infty M_k\]
is called the \emph{origami set} with respect to the slopes given by~$U$. An \emph{origami ring} is an origami set that also is a subring of~$\cc$. The intersection $M_{\rr}(U):=M(U)\cap \rr$ is called the \emph{real part} of~$M(U)$.

\paragraph{Addressing complex numbers}

We present a way of addressing complex numbers that is particularly well suited to describe origami sets.

\noindent\begin{minipage}{.6\textwidth}
Let $\alpha\in\mathopen]0,\pi[$ denote an angle. Since $\alpha$ is non-zero, there is a uniquely defined intersection point of the line $z+\rr\cdot\exp(i\alpha)$ through $z$ with slope~$\alpha$ and the real axis. We denote this point with $\alpha(z)$ and call it the \emph{$\alpha$-projection of $z$}.

For instance, the real part of a complex number $z\in\cc$ is just its $\frac\pi2$-projection.
\end{minipage}
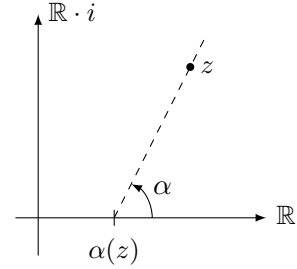
\begin{minipage}[c]{.4\textwidth}
\begin{center}
 \begin{tikzpicture}[>=latex]
    \draw[->] (-.3,0)--(3,0) node[font=\small,right] {\strut$\rr$};
    \draw[->] (0,-.5)--(0,2.7) node[font=\small,right] {\strut$\rr\cdot i$};
    \fill (2,2) circle[radius=1.5pt] node[font=\small,right] {\strut$z$} coordinate (z);
    \draw[dashed] (1,0)--(2.2,2.4);
    \draw (1,3pt)--(1,-3pt) node[font=\small,below] {\strut$\alpha(z)$};
    \coordinate (base) at (1,0);
    \coordinate (x) at (2,0);
    \pic[draw,->,"$\alpha$", angle eccentricity=1.5] {angle = x--base--z};
 \end{tikzpicture}\medskip
\end{center}
\end{minipage}

\noindent It is easily seen that the $\alpha$-projection is a projection in the linear algebraic sense:

\begin{lemma}\label{lem:angleproj}
 Let $\alpha\in\mathopen]0,\pi[$ be an angle. Define the map
  \[\alpha:\;\cc\to\rr,\qquad z\mapsto\alpha(z).\]
 Then the following statements hold:
 \begin{enumerate}
  \item $\alpha$ is $\rr$-linear: The equations  $\alpha(w+z)=\alpha(w)+\alpha(z)$ and $\alpha(\lambda\cdot w)=\lambda\cdot\alpha(w)$ hold for all $w,z\in\cc$ and $\lambda\in\rr$.
  \item $\alpha$ is idempotent: The identity $\alpha\bigl(\alpha(z)\bigr)=\alpha(z)$ holds for all $z\in\cc$.
  \item The restriction $\alpha|_\rr$ is the identity map. The equality $\alpha(x)=x$ holds if and only $x\in\rr$.
 \end{enumerate}
\end{lemma}

\noindent A complex number is uniquely given by two projections:

\begin{definition}
 Let $\alpha,\beta\in\mathopen]0,\pi[$ be two different angles. Then the map
 \[\cc\to\rr^2,\qquad z\mapsto \bigl(\alpha(z),\beta(z)\bigr)\]
 is a bijection. We call the pair~$\bigl(\alpha(z),\beta(z)\bigr)\in\rr^2$ the \emph{${(\alpha,\beta)}$-coordinates} of $z$.
 
 Vice versa, given two real numbers $r,s\in\rr$ we denote with $\llbracket r,s\rrbracket_{\alpha,\beta}$ the unique complex number $z\in\cc$ fulfilling $\alpha(z)=r$ and $\beta(z)=s$. Hence, $\llbracket r,s\rrbracket_{\alpha,\beta}$ is exactly the complex number with $(\alpha,\beta)$-coordinates $(r,s)$ and fulfills the equation
 \[\bigl\{\llbracket r,s\rrbracket_{\alpha,\beta}\bigr\}=\bigl(r+\rr\cdot\exp(i\alpha)\bigr)\cap\bigl(s+\rr\cdot\exp(i\beta)\bigr).\]
 
  \begin{center}
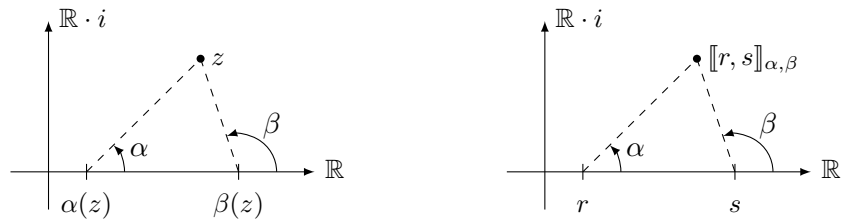

  \begin{tikzpicture}[>=latex]
    \draw[->] (-.5,0)--(3.5,0) node[font=\small,right] {\strut$\rr$};
    \draw[->] (0,-.5)--(0,2) node[font=\small,right] {\strut$\rr\cdot i$};
    \fill (2,1.5) circle[radius=1.5pt] node[font=\small,right] {\strut$z$} coordinate (z);
    \draw[dashed] (.5,0) coordinate (phi)--(z);
    \draw[dashed] (2.5,0) coordinate (psi)--(z);
    \draw (.5,3pt)--(.5,-3pt) node[font=\small,below] {\strut$\alpha(z)$};
    \draw (2.5,3pt)--(2.5,-3pt) node[font=\small,below] {\strut$\beta(z)$};
     \coordinate (x) at (5,0);
     \pic[draw,->,"$\alpha$", angle eccentricity=1.5] {angle = x--phi--z};
     \pic[draw,->,"$\beta$", angle eccentricity=1.5] {angle = x--psi--z};
 \end{tikzpicture}\hspace{2cm}\begin{tikzpicture}[>=latex]
    \draw[->] (-.5,0)--(3.5,0) node[font=\small,right] {\strut$\rr$};
    \draw[->] (0,-.5)--(0,2) node[font=\small,right] {\strut$\rr\cdot i$};
    \fill (2,1.5) circle[radius=1.5pt] node[font=\small,right] {\strut$\llbracket r,s\rrbracket_{\alpha,\beta}$} coordinate (z);
    \draw[dashed] (.5,0) coordinate (phi)--(z);
    \draw[dashed] (2.5,0) coordinate (psi)--(z);
    \draw (.5,3pt)--(.5,-3pt) node[font=\small,below] {\strut$r$};
    \draw (2.5,3pt)--(2.5,-3pt) node[font=\small,below] {\strut$s$};
     \coordinate (x) at (5,0);
     \pic[draw,->,"$\alpha$", angle eccentricity=1.5] {angle = x--phi--z};
     \pic[draw,->,"$\beta$", angle eccentricity=1.5] {angle = x--psi--z};
 \end{tikzpicture}
 \captionof{figure}{The connection between $(\alpha,\beta)$-coordinates and complex numbers.}
 \end{center}
\end{definition}

\noindent The next result follows directly from this definition and Lemma~\ref{lem:angleproj}.

\begin{lemma}\label{lem:anglecoord}
 Let $\alpha,\beta\in\mathopen]0,\pi[$ be two different angles. Then the following statements hold, where we write $\llbracket\cdot,\cdot\rrbracket$ instead of $\llbracket\cdot,\cdot\rrbracket_{\alpha,\beta}$ for the sake of readability:
 \begin{enumerate}
 \item For all $r,s\in\rr$ the equations $\alpha\bigl(\llbracket r,s\rrbracket\bigr)=r$ and $\beta\bigl(\llbracket r,s\rrbracket\bigr)=s$ hold. Conversely, for all $z\in\cc$ one has $z=\llbracket\alpha(z),\beta(z)\rrbracket$.
 \item Let $r,s$ be real numbers. Then, $\llbracket r,s\rrbracket$ is a real number if and only if $r=s$ holds. In this case one has $r=\llbracket r,r\rrbracket$.
  \item As a consequence of~(a) it follows that the map $\llbracket\cdot,\cdot\rrbracket:\;\rr^2\to\cc$ is $\rr$-linear: The equations 
  \[\llbracket r,s\rrbracket+\llbracket r',s'\rrbracket=\llbracket r+r',s+s'\rrbracket\quad\text{and}\quad \llbracket \lambda r,\lambda s\rrbracket=\lambda\,\llbracket r,s\rrbracket\]
  hold for all real numbers $r,s',r',s',\lambda$. Together with~(b) this implies $1=\llbracket1,0\rrbracket+\llbracket0,1\rrbracket$.
 \end{enumerate}
\end{lemma}

\paragraph{Notation} 

Throughout this paper we employ the following notation:

$U\subseteq [0,\pi[$ denotes the set of prescribed slopes. We always assume $0\in U$ and $|U|\geq3$. We write $M$ for the origami set $M(U)$ and $M_\rr$ for its real part. The symbols $\alpha$ and $\beta$ denote angles of~$U\smallsetminus\{0\}$ with $\alpha\ne\beta$. We write $\llbracket\cdot,\cdot\rrbracket$ instead of $\llbracket\cdot,\cdot\rrbracket_{\alpha,\beta}$.

\section{The structure of origami sets}

The aim of this section is twofold:

First, we want to give a set theoretic description of origami rings. This is done in Theorem~\ref{thm:mexplicit} which states that every origami ring is the $M_\rr$-span of~$1$ and $\llbracket1,0\rrbracket$. So the structure of~$M$ is pretty easy — provided that the real part $M_\rr$ of $M$ is known. We deal with this restriction in Theorem~\ref{thm:mrexplicit} where we give an explicit description of $M_\rr$ in terms of the elements of~$U$. A surprising consequence of this theorem is that $M_\rr$ is always a subring of~$\rr$.

Second, we discuss the algebraic structure of~$M$. It is well known that $M$ is an additive group. We prove this in Theorem~\ref{thm:mexplicit}. In Theorem~\ref{thm:ring} we give several criteria for $M$ to be an origami ring.

\subsection*{Reduction to $M_\rr$}

Let $z$  be an element of~$M$. The projections $\alpha(z)$ and $\beta(z)$ are elements of $M_\rr$ since they are intersections of the admissible lines. Conversely, Lemma~\ref{lem:anglecoord} shows that the equation $\alpha(z)=\beta(z)=z$ holds for all $z\in M_\rr$. This gives the equality
\begin{align}\label{eq:mr}
M_\rr=\{\alpha(z),\beta(z) : z\in M\}.
\end{align}
On the other hand, if $r,s$ are elements of $M_\rr$, then $\llbracket r,s\rrbracket$ is an element of $M$. Conversely, by~\eqref{eq:mr}, for any $z\in M$ the $\alpha$- and $\beta$-projections of $z$ are elements of $M_\rr$. Thus,
\begin{align}\label{eq:m}
M=\bigl\{\llbracket r,s\rrbracket : r,s\in M_\rr\bigr\}.
\end{align}
Equation~\eqref{eq:m} shows that~$M$ can be entirely reconstructed out of $M_\rr$. So, no information is lost when transitioning from $M$ to $M_\rr$.\medskip

\noindent It is well known that~$M$ is an additive subgroup of~$\cc$, cf.~\cite[Thm.~3.1]{buh12}. The following preparatory lemma states this result for the real part of $M$:

\begin{lemma}\label{lem:mrgroup}
$M_\rr$ is an additive subgroup of~$\rr$. Since $0,1\in M_\rr$, it follows that $\zz\subseteq M_\rr$.
\end{lemma}

\begin{proof}
 We employ the subgroup test: Let $r,s$ be elements of the non-empty set $M_\rr$ with $s\geq r$. Define $z:=\llbracket r,s\rrbracket$. By definition of the origami set~$M$, the intersection point~$z'$ of the lines $z+\rr\cdot\exp(i0)$ and $0+\rr\cdot\exp(i\alpha)$ is an element of $M$. Its $(\alpha,\beta)$-coordinates are $(0,x)$ where~$x$ denotes an appropriate element of $M_\rr$.
 
 The triangle with vertices $r$, $z$, $s$ is congruent to the triangle with vertices $0$, $z'$, $x$. So, the corresponding sides of both triangles have the same length, giving $s-r=x-0=x\in M_\rr$.
 
 \begin{center}
 
 \begin{tikzpicture}[>=latex]
   \draw[->] (0,0)--(4.5,0) node[font=\small,right] {\strut$\rr$};
   \draw[->] (0,0)--(0,2) node[font=\small,right] {\strut$\rr\cdot i$};
   \draw (2.5,3pt)--(2.5,-3pt) node[font=\small,below] {\strut$r$};
   \draw (3.5,3pt)--(3.5,-3pt) node[font=\small,below] {\strut$s$};
   \draw[black!30] (.5,1.5)--(4.5,1.5);
   \fill (4,1.5) coordinate (z) circle[radius=1.5pt] node[font=\small,above] {\strut$z$};
   \coordinate (right) at (5,0);
   \coordinate (r) at (2.5,0);
   \coordinate (s) at (3.5,0);
   \draw[dashed] (r)--(z);
   \pic[draw,->,"$\alpha$", angle eccentricity=1.5] {angle = right--r--z};
   \draw[dashed] (s)--(z);
   \pic[draw,->,"$\beta$", angle eccentricity=1.5] {angle = right--s--z};
   
   \coordinate (0) at (0,0);
   \coordinate (x) at (1,0);
   \coordinate (zs) at (1.5,1.5);
   
   \fill[black!30] (zs) circle[radius=1.5pt] node[font=\small,above] {\strut$z'$};
   \draw[black!30] (0,0)--(1.8,1.8);
   \pic[black!30,draw,->,"$\alpha$", angle eccentricity=1.5] {angle = right--0--zs};
   
   \draw (0,3pt)--(0,-3pt) node[font=\small,below] {\strut $0$};
  \end{tikzpicture}\qquad\qquad \begin{tikzpicture}[>=latex]
   \draw[->] (0,0)--(4.5,0) node[font=\small,right] {\strut$\rr$};
   \draw[->] (0,0)--(0,2) node[font=\small,right] {\strut$\rr\cdot i$};
   \draw (2.5,3pt)--(2.5,-3pt) node[font=\small,below] {\strut$r$};
   \draw (3.5,3pt)--(3.5,-3pt) node[font=\small,below] {\strut$s$};
   \fill (4,1.5) coordinate (z) circle[radius=1.5pt] node[font=\small,above] {\strut$z$};
   \coordinate (right) at (5,0);
   \coordinate (r) at (2.5,0);
   \coordinate (s) at (3.5,0);
   \draw[dashed] (r)--(z);
   \pic[draw,->,"$\alpha$", angle eccentricity=1.5] {angle = right--r--z};
   \draw[dashed] (s)--(z);
   \pic[draw,->,"$\beta$", angle eccentricity=1.5] {angle = right--s--z};
   
   \coordinate (0) at (0,0);
   \coordinate (x) at (1,0);
   \coordinate (zs) at (1.5,1.5);

   \pic[draw,->,"$\alpha$", angle eccentricity=1.5] {angle = right--0--zs};
   \draw[dashed] (0)--(zs);
   \fill (zs) circle[radius=1.5pt] node[font=\small,above] {\strut$z'$};
   \draw (0,3pt)--(0,-3pt) node[font=\small,below] {\strut $0$};
   \draw (1,3pt)--(1,-3pt) node[font=\small,below] {\strut $x$};
   \pic[draw,->,"$\beta$", angle eccentricity=1.5] {angle = right--x--zs};
   \draw[dashed] (x)--(zs);
  \end{tikzpicture}
  
  
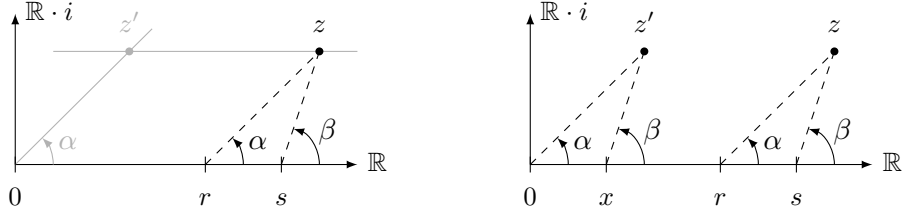
\captionof{figure}{The construction of $z'$. Note that the triangles are congruent. Thus, $s-r=x$.}
  
 \end{center}
 
\noindent By considering the point~$z''$ defined by
\[\{z''\}=\bigl(z+\rr\cdot\exp(i0)\bigr)\cap\bigl(0+\rr\cdot\exp(i\beta)\bigr)\]
 one obtains $r-s\in M_\rr$ in a similar fashion.
\end{proof}

\noindent A consequence of this lemma is the following explicit description of origami sets. We also obtain that origami sets are additive subgroups of~$\cc$.

\begin{theorem}\label{thm:mexplicit}
 $M$ is the $M_\rr$-span of\/ $1$ and $\llbracket0,1\rrbracket$, i.\,e.
 \[M=M_\rr+M_\rr\cdot \llbracket0,1\rrbracket = \bigl\{r+s\cdot \llbracket0,1\rrbracket: r,s\in M_\rr\bigr\}.\]
 Together with Lemma~\ref{lem:mrgroup} this shows that $M$ is an additive subgroup of~$\cc$.
\end{theorem}

\begin{proof}
 Let $z$ be an element of $M$. Then there are $r,s\in M_\rr$ with $z=\llbracket r,s\rrbracket$. Lemma~\ref{lem:anglecoord} gives
 \begin{eqnarray*}
 z&=&\llbracket r,0\rrbracket+\llbracket0,s\rrbracket=r\cdot\llbracket1,0\rrbracket+s\cdot\llbracket0,1\rrbracket\\
 &=&r\cdot\bigl(1-\llbracket0,1\rrbracket\bigr)+s\cdot\llbracket0,1\rrbracket=r\cdot 1 + (s-r)\cdot \llbracket0,1\rrbracket.
 \end{eqnarray*}
 Since $s-r\in M_\rr$ by Lemma~\ref{lem:mrgroup}, one obtains $z\in M_\rr+ M_\rr\cdot \llbracket0,1\rrbracket$.
 
 Conversely assume that $r,s\in M_\rr$. Then, again with Lemma~\ref{lem:anglecoord},
 \[r+s\cdot \llbracket0,1\rrbracket = \llbracket r,r\rrbracket+s\cdot \llbracket0,1\rrbracket=\llbracket r,r\rrbracket+\llbracket 0,s\rrbracket=\llbracket r,r+s\rrbracket.\]
 Since $r+s\in M_\rr$ by Lemma~\ref{lem:mrgroup}, equation~\eqref{eq:m} gives $r+s\cdot \llbracket0,1\rrbracket\in M$.
\end{proof}

\subsection*{The ring structure of ${M_\rr}$}

In this paragraph we show that $M_\rr$ is a subring of~$\rr$. The main technique we employ are coordinate transformations: Given any two angles $\gamma,\delta\in U$, we convert $(\alpha,\beta)$-coordinates into $(\gamma,\delta)$-coordinates and vice versa.\medskip

\noindent The following definition provides a notation that will become handy later on:

\begin{definition}
 Let $\gamma\in U\smallsetminus\{0\}$ be arbitrary. We denote the $\gamma$-projection of $\llbracket0,1\rrbracket$ with
 \[p(\gamma):=\gamma\bigl(\llbracket0,1\rrbracket\bigr).\]
 Equation~\eqref{eq:mr} shows that $p(\gamma)\in M_\rr$, Lemma~\ref{lem:anglecoord} that $p(\alpha)=0$ and $p(\beta)=1$. Note that if $\gamma\ne\delta$, then $p(\gamma)\ne p(\delta)$. Hence, the map $p:\; U\smallsetminus\{0\}\to M_\rr$ is injective.
 \begin{center}
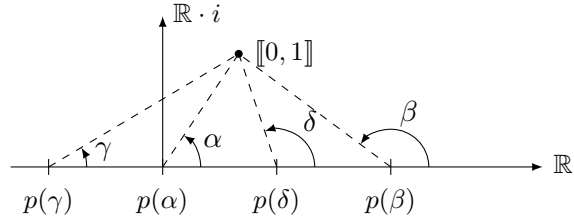

 \begin{tikzpicture}[>=latex]
  \draw[->] (-2,0)--(5,0) node[right,font=\small] {\strut$\rr$};
  \draw[->] (0,0)--(0,2) node[right,font=\small] {\strut$\rr\cdot i$};
  \fill (1,1.5) coordinate (z) circle[radius=1.5pt] node[right=2pt,font=\small] {\strut$\llbracket0,1\rrbracket$};
  \coordinate (right) at (4,0);
  \coordinate (0) at (0,0);
  \coordinate (1) at (3,0);
  \coordinate (gamma) at (-1.5,0);
  \coordinate (delta) at (1.5,0);
  \draw (0,3pt)--(0,-3pt) node[below,font=\small] {\strut$p(\alpha)$};
  \draw (3,3pt)--(3,-3pt) node[below,font=\small] {\strut$p(\beta)$};
  
  \pic[draw,->,"$\alpha$", angle eccentricity=1.5] {angle = right--0--z};
  \draw[dashed] (0)--(z);
   
   \pic[draw,->,"$\beta$", angle eccentricity=1.5] {angle = right--1--z};
   \draw[dashed] (1)--(z);
   
   \pic[draw,->,"$\gamma$", angle eccentricity=1.5] {angle = right--gamma--z};
   \draw[dashed] (gamma)--(z);
   \draw (-1.5,3pt)--(-1.5,-3pt) node[below,font=\small] {\strut$p(\gamma)$};
   
   \pic[draw,->,"$\delta$", angle eccentricity=1.5] {angle = right--delta--z};
   \draw[dashed] (delta)--(z);
   \draw (1.5,3pt)--(1.5,-3pt) node[below,font=\small] {\strut$p(\delta)$};
 \end{tikzpicture}
 \captionof{figure}{Different projections of $\llbracket 0,1\rrbracket$. Note that $p(\alpha)=0$ and $p(\beta)=1$.}
 \end{center}
\end{definition}

\noindent The following result describes how coordinates change when transitioning to different pairs of angles:

\begin{proposition}\label{prop:coordtransform}
 Let $\gamma,\delta\in U\smallsetminus\{0\}$ be two different angles. Then, for any $r,s\in M_\rr$, the following equations hold:
 \begin{enumerate}
  \item $\displaystyle \llbracket r,s\rrbracket=\bigl\llbracket r+(s-r)p(\gamma), r+(s-r)p(\delta)\bigr\rrbracket_{\gamma,\delta}$\,,
  \item $\displaystyle \llbracket r,s\rrbracket_{\gamma,\delta}=\Bigl\llbracket \frac{sp(\gamma)-rp(\delta)}{p(\gamma)-p(\delta)}, \frac{r-s+sp(\gamma)-rp(\delta)}{p(\gamma)-p(\delta)}\Bigr\rrbracket$.
 \end{enumerate}
\end{proposition}

\begin{proof}
\begin{enumerate}
 \item[(a)] We know that $p(\gamma)=\gamma\bigl(\llbracket0,1\rrbracket\bigr)$ and $p(\delta)=\delta\bigl(\llbracket0,1\rrbracket\bigr)$. Therefore we obtain $\llbracket0,1\rrbracket=\llbracket p(\gamma),p(\delta)\rrbracket_{\gamma,\delta}$. By Lemma~\ref{lem:anglecoord},
 \[\llbracket1,0\rrbracket=1-\llbracket0,1\rrbracket=\llbracket1,1\rrbracket_{\gamma,\delta}-\llbracket p(\gamma),p(\delta)\rrbracket_{\gamma,\delta}=\llbracket 1-p(\gamma),1-p(\delta)\rrbracket_{\gamma,\delta}.\]
 The representation of $\llbracket r,s\rrbracket$ in $(\gamma,\delta)$-coordinates is now due to linearity.
 \item[(b)] If $\llbracket r,s\rrbracket_{\gamma,\delta}=\llbracket x,y\rrbracket$ for some $x,y\in M_\rr$, then, by~(a), the variables~$x$ and~$y$ satisfy the linear equation system
 \[\left\{
 \begin{array}{lcl}
    x+(y-x)p(\gamma) &=& r\\
    x+(y-x)p(\delta) &=& s
 \end{array}
 \right\}.\]
 $\gamma\ne\delta$ implies $p(\gamma)\ne p(\delta)$ and, thus, the existence of a unique solution $(x,y)$ of the system. Solving for $x$ and $y$ gives the claim.
 \end{enumerate}
\end{proof}

\noindent The next lemmas show that all differences $p(\gamma)-p(\delta)$ and quotients $\bigl(p(\gamma)-p(\delta)\bigr)^{-1}$ are elements of $M_\rr$. For reasons of readability we introduce the following notation:

\begin{definition}
 We denote by $\Delta$ the set of all differences $p(\gamma)-p(\delta)$ where $\gamma$ and $\delta$ are two different elements of $U\smallsetminus\{0\}$, i.\,e.
 \[\Delta:=\{p(\gamma)-p(\delta) : \gamma,\delta\in U\smallsetminus\{0\}\text{ and } \gamma\ne\delta\}.\]
 It is $0\notin\Delta$. Therefore, we can define
 \[\Delta^{-1}:=\{d^{-1} : d\in \Delta\}=\Bigl\{\frac1{p(\gamma)-p(\delta)} : \gamma,\delta\in U\smallsetminus\{0\}\text{ and } \gamma\ne\delta\Bigr\}.\]
\end{definition}

\begin{lemma}
 Denote the subring of\/~$\rr$ generated by~$\Delta$ with $\zz[\Delta]$. Then $\zz[\Delta]\subseteq M_\rr$.
\end{lemma}

\begin{proof}
 Let~$R$ denote the ring $\zz[p(\gamma):\gamma\in U\smallsetminus\{0\}]$. We first show that $R=\zz[\Delta]$:
 
 It is obvious that $Z[\Delta]\subseteq R$. To prove the opposite inclusion consider an arbitrary element~$r\in R$. By definition, $r$ is a sum of addents of the form
 \[z\cdot p(\gamma_1)\cdots p(\gamma_s)\quad\text{with } z\in\zz\text{ and } \gamma_1,\ldots,\gamma_s\in U\smallsetminus\{0\}.\]
 Since $p(\alpha)=0$ it follows
 \[z\cdot p(\gamma_1)\cdots p(\gamma_s)=z\cdot \bigl(p(\gamma_1)-p(\alpha)\bigr)\cdots\bigl(p(\gamma_s)-p(\alpha)\bigr)\in \zz[\Delta].\]
 This shows that $r\in\zz[\Delta]$.\medskip
 
 \noindent Now we show that $p(\gamma)\cdot M_\rr\subseteq M_\rr$ for all $\gamma\in U\smallsetminus\{0\}$:
 
 Choose $\gamma\in U\smallsetminus\{0\}$ and $s\in M_\rr$ arbitrarily. Then, one has $\llbracket0,s\rrbracket\in M$ and, hence, $\gamma\bigl(\llbracket0,s\rrbracket\bigr)\in M_\rr$. The assertion follows since $\gamma\bigl(\llbracket0,s\rrbracket\bigr)=p(\gamma)\cdot s$ by linearity.\medskip
 
 \noindent This proves the lemma: As $\zz$ is a subset of $M_\rr$, repeatingly applying the above result shows that $M_\rr$ contains all products
 \[z\cdot p(\gamma_1)\cdots p(\gamma_s)\quad\text{with } s\in\nn_0, z\in\zz\text{, and } \gamma_1,\ldots,\gamma_s\in U\smallsetminus\{0\}.\]
 As $M_\rr$ is additively closed, it follows $R\subseteq M_\rr$ and therefore $\zz[\Delta]\subseteq M_\rr$.
\end{proof}

\begin{lemma}\label{lem:zddsubmr}
 Denote the subring of\/~$\rr$ generated by $\Delta\cup\Delta^{-1}$ with $\zz[\Delta,\Delta^{-1}]$. Then $\zz[\Delta,\Delta^{-1}]\subseteq M_\rr$.
\end{lemma}

\begin{proof}
 We already know that $\zz[\Delta]\subseteq M_\rr$. By reasoning in the same fashion as above it suffices to show that $\bigl(p(\gamma)-p(\delta)\bigr)^{-1}\cdot M_\rr\subseteq M_\rr$ for any $\gamma,\delta\in U\smallsetminus\{0\}$ with $\gamma\ne\delta$. This is proved in two steps.\medskip
 
 \noindent First, we show that $p(\gamma)^{-1}\cdot M_\rr\subseteq M_\rr$ holds for any $\gamma\in U\smallsetminus\{0,\alpha\}$:
 
 Let $\gamma\in U\smallsetminus\{0,\alpha\}$ and $r\in M_\rr$ be arbitrary elements. Then $\llbracket r,0\rrbracket_{\gamma,\alpha}\in M$ and, thus, $\beta\bigl(\llbracket r,0\rrbracket_{\gamma,\alpha}\bigr)\in M_\rr$. By Proposition~\ref{prop:coordtransform}~(b),
 \[\beta\bigl(\llbracket r,0\rrbracket_{\gamma,\alpha}\bigr)=\frac{r-rp(\alpha)}{p(\gamma)-p(\alpha)}\stackrel{p(\alpha)=0}=\frac1{p(\gamma)}\cdot r\in M_\rr.\]
 
 \noindent Now, we show $\bigl(p(\gamma)-p(\delta)\bigr)^{-1}\cdot M_\rr\subseteq M_\rr$:
 
 Choose two different elements $\gamma,\delta\in U\smallsetminus\{0\}$ and $s\in M_\rr$ arbitrarily. Suppose that $\gamma\ne\alpha$. Then $p(\gamma)^{-1} s\in M_\rr$ and, thus, $\llbracket0,p(\gamma)^{-1} s\rrbracket_{\gamma,\delta}\in M$. By Proposition~\ref{prop:coordtransform}~(b),
 \[\alpha\bigl(\llbracket0,p(\gamma)^{-1} s\rrbracket_{\gamma,\delta}\bigr)=\frac{p(\gamma)^{-1}s\cdot p(\gamma)}{p(\gamma)-p(\delta)}=\frac1{p(\gamma)-p(\delta)}\cdot s\in M_\rr.\]
 If $\gamma=\alpha$, then $\delta\ne\alpha$. The claim follows similarly by considering~$\alpha\bigl(\llbracket -p(\delta)^{-1} s,0\rrbracket_{\gamma,\delta}\bigr)$.
\end{proof}

\noindent The following theorem explicitly describes the set $M_\rr$:

\begin{theorem}\label{thm:mrexplicit}
 The equality $M_\rr=\zz[\Delta,\Delta^{-1}]$ holds. This shows in particular that $M_\rr$ is a subring of~$\rr$.
\end{theorem}

\begin{proof}
 Due to Lemma~\ref{lem:zddsubmr} we only have to show that $M_\rr\subseteq\zz[\Delta,\Delta^{-1}]$. This is accomplished inductively by showing that the projections $\alpha(M_k)$ and $\beta(M_k)$ are subsets of $\zz[\Delta,\Delta^{-1}]$; here, $M_k$ is defined as in Section~\ref{sec:intro}. Since $M=\bigcup_{k=0}^\infty M_k$, the claim follows from equation~\eqref{eq:mr}.\medskip
 
 \noindent If $k=0$, then $\alpha(M_k)=\beta(M_k)=M_k=\{0,1\}\subseteq\zz[\Delta,\Delta^{-1}]$. Now assume that for some $k\in\nn_0$ both $\alpha(M_k)$ and $\beta(M_k)$ are subsets of $\zz[\Delta,\Delta^{-1}]$. Let $z$ be any element of $M_{k+1}$. Then there exist elements $x,y\in M_k$ and angles $\gamma,\delta\in U$ with $\gamma\ne\delta$ such that
 \[\{z\}=\bigl(x+\rr\cdot\exp(i\gamma)\bigr)\cap\bigl(y+\rr\cdot\exp(i\delta)\bigr).\]
 \noindent Assume first, that both $\gamma\ne0$ and $\delta\ne0$. Then, by Proposition~\ref{prop:coordtransform}~(a),
 \[\gamma(x)=\alpha(x)+\bigl(\beta(x)-\alpha(x)\bigr)p(\gamma).\]
 As $\alpha(x)$ and $\beta(x)$ are elements of $\zz\bigl[\Delta,\Delta^{-1}\bigr]$ by induction, it follows that $\gamma(x)\in\zz[\Delta,\Delta^{-1}]$. A similar argument shows $\delta(y)\in\zz[\Delta,\Delta^{-1}]$. Since $z$ has $(\gamma,\delta)$-coordinates $\bigl(\gamma(x),\delta(y)\bigr)$, it follows from Proposition~\ref{prop:coordtransform}~(b) that
 \begin{eqnarray*}
 \llbracket\alpha(z),\beta(z)\rrbracket&=&z=\llbracket\gamma(x),\delta(y)\rrbracket_{\gamma,\delta}\\
 &=&\left\llbracket\frac{\delta(y)p(\gamma)-\gamma(x)p(\delta)}{p(\gamma)-p(\delta)},\frac{\gamma(x)-\delta(x)+\delta(y)p(\gamma)-\gamma(x)p(\delta)}{p(\gamma)-p(\delta)}\right\rrbracket.
 \end{eqnarray*}
 So $\alpha(z),\beta(z)\in\zz[\Delta,\Delta^{-1}]$.\medskip
 
 \noindent Now assume that $\gamma=0$. This implies $\delta\ne0$. Assume further that $\delta\ne\alpha$; if $\delta=\alpha$, then one can argue analogously using $\beta$ instead of $\alpha$.
 
 By considering $(\alpha,\delta)$-coordinates one obtains $\delta(x),\delta(y)\in\zz[\Delta,\Delta^{-1}]$ in a similar fashion as above. To show that $\alpha(z)\in\zz[\Delta,\Delta^{-1}]$ consider the line $x+\rr$:\smallskip
 
 \noindent\begin{minipage}{.57\textwidth}
  A point $p\in\cc$ lies on $x+\rr$ iff
  \[\alpha(p)-\delta(p)=\alpha(x)-\delta(x)\]
  because the triangle with vertices $\alpha(x)$, $x$, $\delta(x)$ is congruent to the triangle with vertices $\alpha(p)$, $p$, $\delta(p)$.
 \end{minipage}
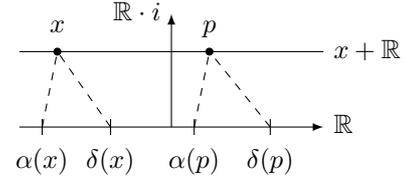
 \begin{minipage}{.43\textwidth}
 \begin{center}
  \begin{tikzpicture}[>=latex]
   \draw[->] (-2,0)--(2,0) node[right,font=\small]{\strut$\rr$};
   \draw[->] (0,0)--(0,1.5) node[left,font=\small]{\strut$\rr\cdot i$};
   \draw (-2,1)--(2,1) node[right,font=\small] {\strut$x+\rr$};
   \coordinate (right) at (4,0);
   
   \fill (-1.5,1) coordinate (x) circle[radius=1.5pt] node[above,font=\small] {\strut$x$};
   \coordinate (ax) at (-1.7,0);
   \draw (-1.7,3pt)--(-1.7,-3pt) node[below,font=\small]{\strut$\alpha(x)$};
   \coordinate (bx) at (-.8,0);
   \draw (-.8,3pt)--(-.8,-3pt) node[below,font=\small]{\strut$\delta(x)$};
   \draw[dashed] (ax)--(x);
   \draw[dashed] (bx)--(x);
   
   \fill (.5,1) coordinate (p) circle[radius=1.5pt] node[above,font=\small] {\strut$p$};
   \coordinate (ap) at (.3,0);
   \draw (.3,3pt)--(.3,-3pt) node[below,font=\small]{\strut$\alpha(p)$};
   \coordinate (bp) at (1.3,0);
   \draw (1.3,3pt)--(1.3,-3pt) node[below,font=\small]{\strut$\delta(p)$};
   \draw[dashed] (ap)--(p);
   \draw[dashed] (bp)--(p);
  \end{tikzpicture}
  \end{center}\smallskip
\end{minipage}

\noindent By induction $\alpha(x)\in\zz[\Delta,\Delta^{-1}]$; thus, $\alpha(x)-\delta(x)\in \zz[\Delta,\Delta^{-1}]$. The definition of~$z$ shows $\delta(z)=\delta(y)$. Since $z\in x+\rr$, it follows
\[\alpha(z)=\underbrace{\delta(y)}_{\in\zz[\Delta,\Delta^{-1}]}+\underbrace{\bigl(\alpha(x)-\delta(x)\bigr)}_{\in\zz[\Delta,\Delta^{-1}]}\in \zz[\Delta,\Delta^{-1}].\]
By transforming the $(\alpha,\delta)$-coordinates of $z$ according to Proposition~\ref{prop:coordtransform}~(b) we finally see
\[\beta(z)=\frac{\alpha(z)-\delta(y)-\alpha(z)p(\delta)}{-p(\delta)}\in\zz[\Delta,\Delta^{-1}].\]

\noindent Due to the symmetry in $\gamma$ and $\delta$, the case $\delta=0$ can be reduced to the case $\gamma=0$. This finishes the proof.
\end{proof}

\begin{example}\label{ex:mr}
 If $U$ contains exactly three elements, then we can write $U=\{0,\alpha,\beta\}$. It follows
 \[\Delta=\bigl\{p(\alpha)-p(\beta),p(\beta)-p(\alpha)\bigr\}=\{1,-1\}\]
 and, subsequently, $\Delta^{-1}=\Delta$. Thus, we end up with $M_\rr=\zz\bigl[\Delta,\Delta^{-1}\bigr]=\zz$.
 
 If $U$ contains exactly four elements, then we can write $U=\{0,\alpha,\beta,\gamma\}$. It follows
 \begin{eqnarray*}
\Delta&=&\bigl\{p(\alpha)-p(\beta),p(\alpha)-p(\gamma),p(\beta)-p(\alpha),p(\beta)-p(\gamma),p(\gamma)-p(\alpha),p(\gamma)-p(\beta)\big\}\\
&=&\bigl\{-1,-p(\gamma),1,1-p(\gamma),p(\gamma),p(\gamma)-1\}.
\end{eqnarray*}
Thus, we obtain $M_\rr=\zz\bigl[\Delta,\Delta^{-1}\bigr]=\zz\bigl[p(\gamma),\frac1{p(\gamma)},\frac1{p(\gamma)-1}\bigr]$.
\end{example}

\subsection*{Origami rings}

\noindent In this paragraph we give several criteria for an origami set to be an origami ring. We start with a technical lemma:

\begin{lemma}\label{lem:explicitcomputations}
 \begin{enumerate}
  \item The equalities
  \[\llbracket0,1\rrbracket=-\frac{\cos\alpha\cdot\sin\beta}{\sin(\alpha-\beta)}-i\cdot \frac{\sin\alpha\sin\beta}{\sin(\alpha-\beta)}\quad\text{and, thus,}\quad \bigl|\llbracket0,1\rrbracket\bigr|^2=\frac{\sin^2\beta}{\sin^2(\alpha-\beta)}\]
  hold where $i$ denotes the imaginary unit.
  \item It is $\llbracket0,1\rrbracket\cdot \llbracket1,0\rrbracket=\left\llbracket\frac{\sin^2\beta}{\sin^2(\alpha-\beta)},\frac{\sin^2\alpha}{\sin^2(\alpha-\beta)}\right\rrbracket$.
  \item For any $\gamma\in U\smallsetminus\{0\}$ one has
  $p(\gamma)=\frac{\sin(\alpha-\gamma)\cdot\sin\beta}{\sin(\alpha-\beta)\cdot\sin\gamma}$.
 \end{enumerate}
 \noindent Note that all of the above quotients are defined: Since $\alpha,\beta,\gamma\in\mathopen]0,\pi[$ with $\alpha\ne\beta$, it follows $\sin(\alpha-\beta)\ne0$ and $\sin\gamma\ne0$.
\end{lemma}

\begin{proof}
 By definition, $\bigl\{\llbracket0,1\rrbracket\bigr\} = \bigl(0+\rr\cdot\exp(i\alpha)\bigr)\cap\bigl(1+\rr\cdot\exp(i\beta)\bigr)$. By solving the equation
 \[\lambda\cdot\exp(i\alpha)=1+\mu\cdot\exp(i\beta) \iff
 \left\{\begin{array}{lcl}
 \lambda\cdot\cos\alpha & = & 1+\mu\cdot\cos\beta\\
 \lambda\cdot\sin\alpha & = & \mu\cdot\sin\beta
 \end{array}\right\}\]
 one obtains the expression for $\llbracket0,1\rrbracket$ given above. The claim about $|\llbracket0,1\rrbracket\bigr|^2$ follows from this and proves~(a).
 
 To show~(b), one uses~(a) and computes $z:=\llbracket0,1\rrbracket\cdot\llbracket1,0\rrbracket=\llbracket0,1\rrbracket\cdot\bigl(1-\llbracket0,1\rrbracket\bigr)$. The intersections
 \[\rr\cap\bigl(z+\rr\cdot\exp(i\alpha)\bigr)\qquad\text{and}\qquad\rr\cap\bigl(z+\rr\cdot\exp(i\beta)\bigr)\]
 now give the $\alpha$- and $\beta$-projection of~$z$, respectively.
 
 The definition of $p(\gamma)$ shows $\{p(\gamma)\}=\rr\cap\bigl(\llbracket0,1\rrbracket+\rr\cdot\exp(i\gamma)$. Using~(a) and solving the system
 \[\left\{\begin{array}{lcl}
 \lambda&=& -\frac{\cos\alpha\cdot\sin\beta}{\sin(\alpha-\beta)}+\mu\cos\gamma\\
0&=& -\frac{\sin\alpha\sin\beta}{\sin(\alpha-\beta)} + \mu\sin\gamma
\end{array}\right\}\]
shows~(c).
\end{proof}

\noindent We can now state our main result:

\begin{theorem}\label{thm:ring}
 For an origami set $M$ the following statements are equivalent:
 \begin{enumerate}
  \item $M$ is an origami ring.
  \item The complex number $\llbracket0,1\rrbracket$ is integral over $M_\rr$ of degree two, i.\,e. there exists a monic irreducible quadratic polynomial $f\in M_\rr[X]$ such that $f\bigl(\llbracket0,1\rrbracket)=0$.
  \item Both $\frac{\sin^2\beta}{\sin^2(\alpha-\beta)}$ and $2\cdot \frac{\cos\alpha\cdot\sin\beta}{\sin(\alpha-\beta)}$ are elements of $M_\rr$.
  \item Both $\frac{\sin^2\alpha}{\sin^2(\alpha-\beta)}$ and $\frac{\sin^2\beta}{\sin^2(\alpha-\beta)}$ are elements of $M_\rr$.
  \item $\llbracket0,1\rrbracket\cdot \llbracket1,0\rrbracket$ is an element of $M$.
 \end{enumerate}
\end{theorem}

\begin{proof}
 For the sake of readability, we set $e:=\llbracket0,1\rrbracket$. Then $\llbracket1,0\rrbracket=1-e$.\smallskip
 
 \noindent Assume~(a). Then $e^2\in M$. By Theorem~\ref{thm:mexplicit}, there exist $r,s\in M_\rr$ with $e^2=r+se$. Thus, the monic quadratic polynomial $f:=X^2-sX-r\in M_\rr[X]$ has~$e$ as a zero. Since~$e$ is not a real number, $f$ is irreducible. This shows~(b).
 
 Assume~(b). Since~$f$ is a real polynomial, the complex conjugate~$\overline e$ of~$e$ is also a zero of~$f$. It follows
 \[f=(X-e)(X-\overline e) = X^2-2\RE(e)\cdot X + |e|^2\in M_\rr[X].\]
 So, $2\RE(e), |e|^2\in M_\rr$. Lemma~\ref{lem:explicitcomputations}~(a) now shows~(c).
 
 Assume~(c). Using the angle difference identities one obtains
 \[\frac{\sin^2\alpha}{\sin^2(\alpha-\beta)}=1+2\frac{\cos\alpha\cdot\sin\beta}{\sin(\alpha-\beta)}+\frac{\sin^2\beta}{\sin^2(\alpha-\beta)}.\]
 Since $M_\rr$ is additively closed, this shows $\frac{\sin^2\alpha}{\sin^2(\alpha-\beta)}\in M_\rr$ and proves~(d).
 
 Assume~(d). Then the $(\alpha,\beta)$-coordinates of $e(1-e)$ are elements of $M_\rr$ by Lemma~\ref{lem:explicitcomputations}~(b). This shows $e(1-e)\in M$ and proves~(e).
 
 Assume~(e). Since $e\in M$, the additive group structure of $M$ gives
 \[e^2=e-e+e^2=e-e(1-e)\in M.\]
 Therefore there are $r,s\in M_\rr$ such that $e^2=\llbracket r,s\rrbracket$. To prove that $M$ is a subring of~$\cc$ we only have to show that $M$ is multiplicatively closed. To this end, let $x,y$ be arbitrary elements of~$M$. By Theorem~\ref{thm:mexplicit}, there exist $a,b,c,d\in M_\rr$ such that $x=a+be$ and $y=c+de$. It follows
 \begin{eqnarray*}
 xy&=&ac + (ad+bc)\cdot e+bd\cdot e^2 = ac\cdot\llbracket1,1\rrbracket+(ad+bc)\cdot\llbracket0,1\rrbracket+bd\cdot\llbracket r,s\rrbracket\\
 &\stackrel{\text{Linearity}}=& \llbracket ac+bdr, ac+ad+bc+bds\rrbracket.
 \end{eqnarray*}
 Due to the ring structure of $M_\rr$ both the $\alpha$- and the $\beta$-projection of $xy$ are elements of~$M_\rr$. This shows $xy\in M$ and proves~(a).
\end{proof}

\noindent We discuss the situation when additional angles are added to the set~$U$:

If the set $U'\subseteq[0,\pi[$ contains~$U$, then obviously $M\subseteq M(U')$. Choosing $\alpha,\beta\in U$ and using criterion~(e) of Theorem~\ref{thm:ring} readily yields

\begin{corollary}\label{cor:ringstable}
 If $M$ is an origami ring and if $U'\subseteq[0,\pi[$ contains $U$, then $M(U')$ is an origami ring, too. Hence, the ring property of $M$ is preserved under extensions of~$U$.
\end{corollary}

\noindent The next result deals with the question whether every origami set  is a subset of an origami ring:

\begin{corollary}
 If the set~$U$ of prescribed slopes contains $\frac\pi3$ and $\frac{2\pi}3$, then $M$ is an origami ring.
 
 This and Corollary~\ref{cor:ringstable} show: By allowing at most two additional slopes, every origami ring “extends” to an origami ring. In particular, every origami set is contained in an origami ring.
\end{corollary}

\begin{proof}
 Set $\alpha:=\frac\pi3$ and $\beta:=\frac{2\pi}3$. Then
 \[\frac{\sin^2\alpha}{\sin^2\bigl(\alpha-\beta\bigr)}=1=\frac{\sin^2\beta}{\sin^2\bigl(\alpha-\beta\bigr)}.\]
 Hence, $M\bigl(U\cup\{\alpha,\beta\}\bigr)$ is an origami ring by Theorem~\ref{thm:ring}~(d).
\end{proof}

\begin{remark}
 There exist sets $U$ such that, regardless of the choice of $\gamma\in[0,\pi[$, the origami set $M\bigl(U\cup\{\gamma\}\bigr)$ is not an origami ring. Thus, in the general case we cannot expect a one-angle version of the above corollary.\medskip
 
 \noindent  An example of such a set~$U$ can be found by constructing angles $\alpha,\beta\in\mathopen]0,\pi[$ that fulfill the conditions
 \[\text{(I)}\;\;\frac{\sin^2\alpha}{\sin^2\bigl(\alpha-\beta\bigr)}=\sqrt2\qquad\text{and}\qquad \text{(II)}\;\; \frac{\sin^2\beta}{\sin^2\bigl(\alpha-\beta\bigr)}\text{ is transcendental over $\qq$}.\]
 To this end, let $\alpha\in\mathopen]0,\pi[$ be arbitrary. Set $\beta:=\alpha-\arcsin\Bigl(\frac1{\sqrt[4]2}\cdot\sin\alpha\Bigr)$ and consider $\beta$ a function with respect to~$\alpha$. Since the derivative of~$\beta$ is always positive, it follows that $\beta\in\mathopen]0,\pi[$. Moreover, this choice of~$\beta$ fulfills equation~(I) and always gives $\alpha\ne\beta$. Now, consider the continuous function
 \[f:\; ]0,\pi[\to\rr,\qquad \alpha\mapsto \frac{\sin^2\beta}{\sin^2\bigl(\alpha-\beta\bigr)}.\]
 As $f$ is non-constant, there exists $\alpha\in\mathopen]0,\pi[$ such that $f(\alpha)$ is transcendental over $\qq$. Thus, we have found angles $\alpha,\beta\in\mathopen]0,\pi[$ with $\alpha\ne\beta$ that fulfill (I) and (II).
  
 Set $U:=\{0,\alpha,\beta\}$. Example~\ref{ex:mr} shows that $M(U)_\rr=\zz$. So $M(U)$ is not an origami ring by Theorem~\ref{thm:ring}~(d). Let $\gamma\in[0,\pi[$ be arbitrary. We may assume that $\gamma\notin U$. Then
 \[M\bigl(U\cup\{\gamma\}\bigr)_\rr\stackrel{\text{Ex.~\ref{ex:mr}}}=\zz[p(\gamma),\frac1{p(\gamma)},\frac1{p(\gamma)-1}\bigr]\subseteq\qq\bigl(p(\gamma)\bigr),\]
 where $\qq\bigl(p(\gamma)\bigr)$ denotes the subfield of~$\rr$ generated by $p(\gamma)$.
 
 We now show that $\gamma$ cannot be chosen in a way such that $M\bigl(U\cup\{\gamma\}\bigr)$ becomes an origami ring. To this end assume that Theorem~\ref{thm:ring}~(d) is fulfilled. Then, 
 by (II), $\qq\bigl(p(\gamma)\bigr)$ contains a transcendental element. Hence, $p(\gamma)$ must be transcendental. But then $\qq\bigl(p(\gamma)\bigr)$ is a purely transcendental extension of $\qq$ that does not contain~$\sqrt2=\frac{\sin^2\alpha}{\sin^2(\alpha-\beta)}$.
\end{remark}

\section{Some Examples}

\noindent In this section we further illustrate our notation and results by giving examples of origami sets and rings.

\paragraph{Discrete origami sets} It is well known that an origami set $M(U)$ is dense in~$\cc$ if and only if $U$ contains more than three elements, cf.~\cite[Cor.~10]{ned15} or~\cite[Thm.~3.7]{bah16}. This fact also follows from the proof of our next result which classifies the origami sets that are discrete subsets of~$\cc$.

\begin{proposition}
 For an origami set $M:=M(U)$ the following statements are equivalent:
 \begin{enumerate}
  \item $M$ is a discrete subset of~$\cc$.
  \item $|U|=3$.
  \item $M_\rr=\zz$.
  \item There exists an element $z\in M$ such that $M=\zz+z\cdot\zz$.
 \end{enumerate}
\end{proposition}

\begin{proof}
 Assume that~(b) does not hold. Then $|U|\geq4$, and we can write $U=\{0,\alpha,\beta,\gamma,\ldots\}$ with pairwise different elements $\alpha,\beta,\gamma\in\mathopen]0,\pi[$. By considering $p(\gamma)$ or $p(\gamma)^{-1}$ we see that $M_\rr\cap\mathopen]0,1[\ne\varnothing$. The ring structure of $M_\rr$ shows that $0\in M_\rr$ is a limit point of $M_\rr$. Since $M_\rr$ is a subgroup of~$\rr$, it follows that $M_\rr$ ist dense in~$\rr$. The map $\llbracket\cdot,\cdot\rrbracket:\;\rr^2\to\cc$ is surjective and, as it is linear, continuous. Hence, $M=\llbracket M_\rr,M_\rr\rrbracket$ is a dense subset of~$\cc$. Therefore~(a) does not hold.
 
 Assume~(b). Then~(c) follows from Example~\ref{ex:mr}.
 
 Assume~(c) and set $z:=\llbracket 0,1\rrbracket$. Then Theorem~\ref{thm:mexplicit} gives $M=M_\rr+z\cdot M_\rr=\zz+z\cdot\zz$.
 
 Assume~(d). Then $M$ is a lattice and, thus, a discrete subset of~$\cc$. So,~(a) holds.
\end{proof}

\begin{remark}
 Discrete origami sets were studied in detail by Nedrenco~\cite{ned15}. In this special case his Theorem~$2$ corresponds to our Theorem~\ref{thm:mexplicit} and his Remark~$4$ matches criterion~(c) of our~Theorem~\ref{thm:ring}.
\end{remark}

\paragraph{An origami ring with $|U|=4$}

Bahr and Roth~\cite{bah16} consider the origami set $M:=M(U)$ with $U=\bigl\{0,\frac\pi3,\frac\pi4,\frac\pi5\}$. They ask whether $M$ is an origami ring and strongly suspect that it is not. In fact, it is:\medskip

\noindent Set $\alpha:=\frac\pi3$, $\beta:=\frac\pi4$, and $\gamma:=\frac\pi5$. One obtains
\[\frac{\sin^2\alpha}{\sin^2(\alpha-\beta)}=6+3\sqrt3 \qquad\text{and}\qquad \frac{\sin^2\beta}{\sin^2(\alpha-\beta)}=4+2\sqrt3.\]
It readily follows that $M$ is an origami ring if and only if $\sqrt3\in M_\rr=\zz\Bigl[p,\frac1p,\frac1{p-1}\Bigr]$ where
\[p:=p(\gamma)=\frac{\sin(\alpha-\gamma)\cdot\sin\beta}{\sin(\alpha-\beta)\cdot\sin\gamma}=\bigl(1+\sqrt3\bigr)\cdot\sqrt{2+\frac2{\sqrt5}}\cdot\sin\Bigl(\frac{2\pi}{15}\Bigr).\]
Note that $p$ is algebraic over~$\qq$; its minimal polynomial is given by
\[\mu=X^8+4X^7-8X^6-20X^5+\frac{104}5X^4+16X^3-8X^2-\frac{16}5X+\frac{16}{25}\in\qq[X].\]
Since $M_\rr\subseteq\qq(p)$, a necessary condition for $M$ to be a ring is $\sqrt3\in\qq(p)$. This can be checked with a computational algebra system such as MAGMA~\cite{magma}. One obtains that the polynomial $X^2-3\in\qq(p)[X]$ splits, showing that the necessary condition is fulfilled.\medskip

\noindent Now we show that even $\sqrt3\in M_\rr$ holds. Since
\[M_\rr=\zz\bigl[p,\frac1p,\frac1{p-1}\Bigr]=\Bigl\{\frac{f(p)}{p^a\cdot(p-1)^b} : f\in\zz[X] \text{ and } a,b\in\nn_0\Bigr\},\]
$\sqrt3$ is an element of $M_\rr$ if and only if there are $f\in\zz[X]$ and $a,b\in\nn_0$ such that the equation $\sqrt3\cdot p^a (p-1)^b=f(p)$ is fulfilled.

We briefly sketch how $f$, $a$, and $b$ can be found: Choose random parameters $a,b$ and check if the polynomial $X^2-3p^{2a}(p-1)^{2b}$ splits over $\qq(p)$. If it does, its roots can be represented as polynomials in~$p$ with rational coefficients. We are only interested in the case where the denominators of these coefficients are~$1$, or~$5$, or~$25$. This case occurs, for instance, for $(a,b)=(5,4)$ and gives
\[\sqrt3\cdot p^5(p-1)^4=80p^7-82p^6-\frac{1573}5 p^5+\frac{1278}5 p^4+\frac{1224}5 p^3-\frac{524}5 p^2-\frac{1208}{25} p+\frac{232}{25}.\]
Now, $\mu$ can be used to get rid of these denominators. We explain the procedure with the aid of the coefficient~$\frac{232}{25}$. Note that $\frac{232}{25}+23\cdot\frac{16}{25}=24\in\zz$. Thus, we can “lift” the coefficient~$\frac{232}5$ into the set of integers by computing
\begin{eqnarray*}
 \sqrt3\cdot p^5(p-1)^4&=&\sqrt3\cdot p^5(p-1)^4+23\cdot 0=\sqrt3\cdot p^5(p-1)^4+23\cdot\mu(p)\\
 &=&23 p^8+172p^7-266p^6-\frac{3873}5 p^5+734p^4\\
 &&\quad+\frac{3064}5 p^3-\frac{1444}5 p^2-\frac{3048}{25}p+24.
\end{eqnarray*}
Using this technique several times, one finally gets the equation
\begin{eqnarray*}
\sqrt3&=&\frac1{p^5\cdot(p-1)^4}\cdot\Bigl(-20p^{13}-80p^{12}+140p^{11}+305p^{10}-338p^9+110 p^8+292p^7\\
&&\quad-194 p^6-825 p^5+46p^4+424p^3-28p^2-56p+8\Bigr)\in M_\rr.
\end{eqnarray*}
This equality can now easily be verified with the help of a technical computing system such as Mathematica. By Theorem~\ref{thm:ring}~(d), $M$ is an origami ring.

\newpage 
\bibliographystyle{plainnat}
\bibliography{origamisets} 

\end{document}